\documentclass[letter, 11 pt]{amsart}

\usepackage[hyphens]{url} 
\usepackage{amsthm}
\usepackage{graphicx}
\usepackage{bbm}
\usepackage[english]{babel}
\usepackage{fancyhdr}
\usepackage[T1]{fontenc}
\usepackage[utf8]{inputenc}
\usepackage{mathtools}
\usepackage{amsmath,amsfonts}
\usepackage{enumitem}
\usepackage{mathrsfs}
\usepackage{textcomp}
\usepackage{tikz-cd}
\usepackage{polski}
\usepackage{array}
\usepackage{amssymb}
\usepackage{thmtools}
\usepackage[new]{old-arrows}
\usepackage[bindingoffset=0.2in,%
            left=1.3in,right=1.3in,top=1.2in,bottom=1.2in,%
            footskip=.25in]{geometry}

\graphicspath{ {./images/} }

\theoremstyle{plain}
\newtheorem{thm}{Theorem}[section]
\newtheorem{prop}[thm]{Proposition}

\newtheorem{lem}[thm]{Lemma}
\newtheorem{cor}[thm]{Corollary}

\newtheorem{dfn}[thm]{Definition}
\newtheorem{rem}[thm]{Remark}
\numberwithin{equation}{section}
\newtheoremstyle{sltheorem}
{}                
{}                
{\upshape}        
{}                
{\bfseries}       
{.}               
{ }               
{}                
\theoremstyle{sltheorem}
\newtheorem{exm}{Example}[section]

\newcommand{\R}{\mathbb{R}}
\newcommand{\C}{\mathbb{C}}
\newcommand{\Z}{\mathbb{Z}}

\newcommand{\M}{\mathcal{M}}

\newcommand{\s}{\mathfrak{s}}

\newcommand{\SU}{\mathfrak{su}(2)}

\newcommand{\G}{\mathcal{G}}
\newcommand{\A}{A^{\text{t}}}
\newcommand{\B}{B^{\text{t}}}

\DeclareMathOperator{\im}{im}

\DeclareMathOperator{\supp}{supp}

\DeclareMathOperator{\Met}{Met}
\DeclareMathOperator{\id}{id}
\DeclareMathOperator{\Int}{Int}

\DeclareMathOperator{\spin}{spin}
\DeclareMathOperator{\ind}{ind}

\DeclareMathOperator{\sign}{sign}

\DeclareMathOperator{\tr}{tr}

\DeclareMathOperator{\End}{End}

\DeclareMathOperator{\grad}{grad}
\DeclareMathOperator{\gr}{gr}

\DeclareMathOperator{\cl}{cl}
\DeclareMathOperator{\Red}{red}

\DeclareMathOperator{\PDL}{PD}

\mathchardef\mhyphen="2D

\begin{document}
\title[Fiber Sum for $\lambda_{SW}$]{Fiber Sum Formulae for the Casson-Seiberg-Witten Invariant of Integral Homology $S^1 \times S^3$}
\author{Langte Ma}
\address{MS 050 Department of Mathematics, Brandeis University, 415 South St., Waltham MA 02453}
\email{ltmafixer@brandeis.edu}

\maketitle

\begin{abstract}
We prove the additivity of the Casson-Seiberg-Witten invariant of integral homology $S^1 \times S^3$ under fiber sum along embedded curves and embedded tori, which is the $4$-dimensional analogue of the additivity of the Casson invariant under connected-sum and splicing along knots. 
\end{abstract}
\section{Introduction}

 In \cite{Lim1} Lim interpreted the Casson invariant for an integral homology sphere as the counting of irreducible monopoles corrected by the eta invariants of the Dirac operator and signature operator. Following the same scheme Mrowka-Ruberman-Saveliev introduced the Casson-Seiberg-Witten invariant $\lambda_{SW}$ for an integral homology $S^1 \times S^3$ as a $4$-dimensional analogue of the Casson invariant in \cite{MRS}. One of the prominent feature of the Casson invariant is that it interacts nicely with respect to topological operations, i.e. surgery, connected-sum, splicing etc. In this paper we would like to prove the analogous properties of the Casson-Seiberg-Witten invariant corresponding to splicing and connected-sum in the $3$-dimensional case, which we refer to as fiber sums.

More precisely let $(X_1, \mathcal{T}_1)$ and $(X_2, \mathcal{T}_2)$ be two sets of data such that $X_i$ is an integral homology $S^1 \times S^3$, $\mathcal{T}_i \subset X_i$ is an embedded torus with the map $H_1(\mathcal{T}_i; \Z) \to H_1(X_i; \Z)$ a surjection, $i=1, 2$. After fixing a framing of $\mathcal{T}_i$, i.e. an identification of a tubular neighborhood $\nu(\mathcal{T}_i)$ of $\mathcal{T}_i$ as $D^2 \times T^2$, we get a basis $\{ \mu_i, \lambda_i, \gamma_i\}$ for $H_1(\partial \nu(\mathcal{T}_i); \Z)$, where $\gamma_i$ is chosen to represent a generator of $H_1(X_i; \Z)$. Then the fiber sum of $(X_1, \mathcal{T}_1)$ and $(X_2, \mathcal{T}_2)$ is obtained by gluing the complement $X_1 \backslash \nu(\mathcal{T}_1)$ and $ X_2 \backslash \nu(\mathcal{T}_2)$ using the diffeomorphism on $T^3$ represented by the matrix
\[
\varphi_{\mathcal{T}}=
\begin{pmatrix}
0 & 1 & 0 \\
1 & 0 & 0 \\
0 & 0 & 1
\end{pmatrix}
\]
under the basis $\{\mu_i, \lambda_i, \gamma_i\}$. We will give a more detailed description of the construction in Section \ref{set-up}. 

\begin{thm}\label{main1}
The Casson-Seiberg-Witten invariant is additive under fiber sum along tori, i.e.
\begin{equation}
\lambda_{SW}(X)= \lambda_{SW}(X_1) + \lambda_{SW}(X_2),
\end{equation}
where $X$ is the fiber sum of $X_1$ and $X_2$ along $\mathcal{T}_1$ and $\mathcal{T}_2$.
\end{thm}

The fiber sum along curves are defined as follows. Let $(X_1, \gamma_1)$ and $(X_2, \gamma_2)$ be two sets of data such that $X_i$ is an integral homology $S^1 \times S^3$, $\gamma_i \hookrightarrow X_i$ is an embedded simple closed curve such that $[\gamma_i]$ generates $H_1(X_i; \Z)$. We denote by $M_i$ the closure of the complement of a tubular neighborhood of $\gamma_i$ in $X_i$. The fiber sum $X$ of $X_1$ and $X_2$ is obtained by gluing together $M_1$ and $M_2$ using the identity map after fixing framings of the neighborhoods of $\gamma_i$.   

\begin{thm}\label{main2}
The Casson-Seiberg-Witten invariant is additive under fiber sum along curves, i.e.
\begin{equation}
\lambda_{SW}(X)= \lambda_{SW}(X_1) + \lambda_{SW}(X_2),
\end{equation}
where $X$ is the fiber sum of $X_1$ and $X_2$ along $\gamma_1$ and $\gamma_2$.
\end{thm}

In the product case when $X_i=S^1 \times Y_i$ with $Y_i$ an integral homology sphere, the Casson-Seiberg-Witten invariant $\lambda_{SW}(X_i)$ reduces to the Casson invariant $-\lambda(Y_i)$ (c.f. \cite{MRS}). If we take knots $K_i \subset Y_i$, we get embedded tori $\mathcal{T}_i=S^1 \times K_i \subset X_i$. The fiber sum of $(X_1, \mathcal{T}_1)$ and $(X_2, \mathcal{T}_2)$ is the product $S^1 \times Y_1 \#_{K} Y_2$, where $ Y_1 \#_{K} Y_2$ is the splicing of $Y_1$ and $Y_2$ along knots $K_1$ and $K_2$. Thus Theorem \ref{main1} recovers the additivity of the Casson invariant under splicing along knots. If we fix a point $y_i \in Y_i$, and choose $\gamma_i=S^1 \times \{y_i\}$, the fiber sum of $X_1$ and $X_2$ along $\gamma_i$ is $X=S^1 \times (Y_1 \# Y_2)$, where the connected sum takes place at $y_1$ and $y_2$. In this viewpoint Theorem \ref{main2} recovers the additivity of the Casson invariant. 

The paper is organized as follows. In Section \ref{pre} we recall necessary background Seiberg-Witten theory for the proof. Then in Section \ref{CIM} we make use of the machinery developed by Kronheimer-Mrowka in \cite{KM1} to prove the additivity of counting for irreducible monopoles by running the neck-stretching argument. Section \ref{CICT} uses the technique developed in \cite{M1} to prove the additivity of the correction term in the definition of the Casson-Seiberg-Witten invariant. Finally in Section \ref{app} we present some examples where the fiber sums arise in integral homology $S^1 \times S^3$. 

\subsection*{Acknowledgement} The author would like to express his gratitude to Daniel Ruberman, Youlin Li, Jianfeng Lin, and McKee Krumpak for generously sharing their expertise. 

\section{Preliminaries}\label{pre}

\subsection{The Casson-Seiberg-Witten Invariant}
We briefly review the definition of the Casson-Seiberg-Witten invariant $\lambda_{SW}$. For more details, one should consult the original paper  \cite{MRS}. 

Let $(X, g)$ be a Riemannian smooth closed oriented $4$-manifold with the same integral homology as that of $S^1 \times S^3$, i.e. $H_*(X; \Z) \cong H_1(S^1 \times S^3; \Z)$. Let $\s=(W, \rho)$ be the unique spin$^c$ structure on $X$, where $\rho: T^*X \to \End(W)$ is the Clifford multiplication, and $W=W^+ \oplus W^-$ is a $\C^4$-bundle over $X$. Let's fix an positive integer $k \geq 2$. We write $\mathcal{A}_k(X, \s)$ for the set of $L^2_k$ spin$^c$ connections on $W$. The blown-up configuration space is 
\begin{equation}
\mathcal{C}_k^{\sigma}(X, \s):=\{(A, s, \phi) \in \mathcal{A}_k(X, \s) \times \R \times L^2_k(W^+): s \geq 0, \|\phi\|_{L^2}=1\}.
\end{equation}
The gauge group is $\mathcal{G}_{k+1}(X, \s):= L^2_{k+1}(X, S^1)$ consisting of $L^2_{k+1}$-maps from $X$ to $S^1$, whose action on $\mathcal{C}^{\sigma}_k$ is given by 
\begin{equation}
u \cdot (A, s, \phi) = (A- u^{-1}du \otimes 1_{W^+}, s, u \phi).
\end{equation}
The quotient configuration space is denoted by $\mathcal{B}^{\sigma}_{k}(X, \s):= \mathcal{C}^{\sigma}_k / \G_{k+1}$. Let $\mathcal{P}(X):=L^2_k(T^*X \otimes i\R)$ be the space of perturbations. The Seiberg-Witten map is defined as 
\begin{equation}
\begin{split}
\mathfrak{F}^{\sigma}_{\beta}: \mathcal{C}^{\sigma}_k & \longrightarrow L^2_{k-1}(\Lambda^+ T^*X \otimes i\R) \oplus L^2_{k-1}(X, W^-) \\
(A, s, \phi) & \longmapsto ({1 \over 2}F^+_{\A} - s^2\rho^{-1}(\phi\phi^*)_0 -2d^+\beta, D^+_A \phi),
\end{split}
\end{equation}
where $\A$ is the connection on $\det W^+$ induced by $A$, $(\phi \phi^*)_0 = \phi \otimes \phi^* - {1 \over 2}\tr( \phi \otimes  \phi^*) \in i \SU(W^+)$, and $D^+_A: L^2_{k-1}(W^+) \to L^2_{k-1}(W^-)$ is the Dirac operator. The blown-up moduli space
\begin{equation}
\M_{g, \beta}(X, \s):= \{ [A, s, \phi] \in \mathcal{B}^{\sigma}_k: \mathfrak{F}^{\sigma}_{\beta}(A, s, \phi)=0 \}
\end{equation}
is an oriented compact $0$-manifold for a generic pair $(g, \beta)$ (c.f. \cite{MRS}). The counting of the moduli space $\# \M_{g, \beta}(X, \s)$ forms the first part of the definition of $\lambda_{SW}(X)$. 

The second part of $\lambda_{SW}(X)$ is given by an index correction term $\omega(X, g, \beta)$ defined as follows. We fix a generator $1_X \in H^1(X; \Z)$ which is represented by a smooth map $f: X \to S^1$. Let $Y \subset X$ be an embedded hypersurface with $[Y] = \PDL 1_X \in H_3(X; \Z)$. $Y$ is called a generating hypersurface of $X$. Note that the unique spin$^c$ structure on $X$ is given by spin structures on $X$, of which we fix one now. We denote by $\mathfrak{t}$ the spin structure on $Y$ induced from $\s$. Cutting $X$ along $Y$ results in a spin cobordism $W: Y \to Y$. Let $(Z, \s)$ be an arbitrary spin $4$-manifold with spin boundary $(Y, \mathfrak{t})$. Then we form a spin $4$-manifold with a periodic end: 
\[
Z_+(X):= Z \cup W_0 \cup W_1 \cup ...,
\]
where each $W_i$ is a copy of $W$. Lifting and extending the pair $(g, \beta)$ to $Z_+$ arbitrarily, we can consider the twisted Dirac operator on $Z_+$: 
\begin{equation}\label{pdo}
D^+_{\beta}(Z, g):= D^+(Z_+, \s, g)+\rho(\beta): L^2_1(Z_+, W^+) \to L^2(Z_+, W^-). 
\end{equation}

\begin{dfn}\label{R1}
 We call a pair $(g, \beta) \in \Met(X) \times \mathcal{P}(X)$ regular if the family of Dirac operators 
\begin{equation}\label{1.6.1}
D^+_{z, \beta}(X, \s): =D^+(X, \s) +\rho(\beta- \ln z \cdot df), \; |z|=1,
\end{equation}
have trivial kernel. 
\end{dfn}

\begin{rem}\label{reg} 
Let's fix a metric $g$ on $X$. $\{ D^+_{z, \beta}(X, \s) \}$ forms an $S^1$-family of complex Fredholm operators of index $0$. The stratum consisting of operators with nontrivial kernel in this space has real codimension $2$ (c.f. \cite{K}). Thus for a generic choice of perturbation $\beta$, this family consists of invertible operators. For this reason given any metric $g$, varying perturbations suffices to give us regular pairs $(g, \beta)$. 
\end{rem}

Here $\ln z$ is defined by choosing a branch. For different choice of branches, the above operators in (\ref{1.6.1}) differ by a conjugation of $e^{2\pi k f}$. Thus it is well-defined for $(g, \beta)$ to be regular. It's proved in \cite[Proposition 2.2]{MRS} that being regular is a generic property. Theorem 3.1 in \cite{MRS} says that for any regular pair $(g, \beta)$ the twisted Dirac operator $D^+_{\beta}(Z, g)$ is Fredholm, thus has index defined. We let 
\begin{equation}
\omega(X, g, \beta):=\ind_{\C}D^+_{\beta}(Z_+, g)+{\sigma(Z) \over 8},
\end{equation}
where $\sigma(Z)$ is the signature of $Z$, be the index correction term. Then the Casson-Seiberg-Witten invariant for $X$ is defined as 
\begin{equation}
\lambda_{SW}(X):= \# \M_{g, \beta}(X, \s) - \omega(X, g, \beta),
\end{equation}
where $(g, \beta)$ is a regular pair. One of the main results of \cite{MRS} is that $\lambda_{SW}(X)$ is independent of the choice of the regular pair $(g, \beta)$. We note here that if we change the sign of the generator $1_X \in H^1(X;\Z)$, $\lambda_{SW}(X)$ switches its sign. 

\subsection{Fiber Sums}\label{set-up}
Here we give a detailed description of the fiber sums. 

\subsubsection{Fiber Sum along Tori}

Let $\iota: T^2 \hookrightarrow X$ be an embedding of a torus in a integral homology $S^1 \times S^3$ with the property that the induced map $\iota_*: H_1(T^2; \Z) \to H_1(X; \Z)$ is surjective. Denote by $\mathcal{T} :=\im \iota$ the image of $T^2$ in $X$, which is a trivial disk bundle over $T^2$ due to the vanishing of the intersection form of $X$. Let $M=\cl (X \backslash \nu(\mathcal{T}))$ be the closure of the complement of $\nu(T)$. Under a fixed framing $\nu(\mathcal{T}) \cong D^2 \times T^2$ we get a canonical choice of a triple of curves which form a basis of $H_1(\partial \nu(\mathcal{T}) ;\Z)$: 
\[
\mu :=\partial D^2 \times \{pt. \} \times \{ pt.\}, \lambda :=\{pt. \} \times S^1 \times \{pt.\}, \gamma:=\{pt\} \times \{pt.\} \times S^1.
\]
We choose a framing so that 
\begin{equation}\label{fram}
[\lambda] \in \ker i_* \text{ and } 1_X \cdot [\gamma] =1,
\end{equation}
where $i_*: H_1(\partial \nu(\mathcal{T}); \Z) \to H_1(M; \Z)$ is induced by the inclusion map. 

\begin{dfn}\label{dfn1}
Let $(X_1, \mathcal{T}_1)$ and $(X_2, \mathcal{T}_2)$ be two sets of data as above with fixed framings satisfying (\ref{fram}). The fiber sum of $(X_1, \mathcal{T}_1)$ and $(X_2, \mathcal{T}_2)$ is the manifold given by 
\[
X_1 \#_{\mathcal{T}} X_2=M_1 \cup_{\varphi_{\mathcal{T}}} M_2,
\]
where $\varphi_{\mathcal{T}}: \partial M_2 \to \partial M_1$ is an orientation-reversing diffeomorphism whose isotopy class in $-SL(3; \Z)$ is given by the matrix
\[
\varphi_{\mathcal{T}}=
\begin{pmatrix}
0 & 1 & 0 \\
1 & 0 & 0 \\
0 & 0 & 1
\end{pmatrix}
\]
under the basis $\{\mu_i, \lambda_i, \gamma_i\}$ of $H_1(\partial \nu(\mathcal{T}_i); \Z)$, $i=1, 2$. 
\end{dfn}

\begin{rem}
In practice we will write $X_1=M_1 \cup N_1$, $X_2=N_2 \cup M_2$, which means we identify $\partial M_1=-\partial N_1=T^3$, $\partial N_2 = -\partial M_2=T^3$. Thus the basis $\{ \mu_1, \lambda_1, \gamma_1\}$ of $\partial N_1$ becomes $\{ \lambda_1, \mu_1, \gamma_1\}$ for $\partial M_1$. In this way the gluing map $\varphi_{\mathcal{T}}$ is represented by the identity matrix under the new basis of $\partial M_1$ and $\partial N_2$. When we talk about monopole Floer homology of $T^3$ in Proposition \ref{cirr} , it's always identified with the copy $\partial M_1$ instead of $\partial N_2$.
\end{rem}

For $i=1, 2$, followed from Lemma 2.4 in \cite{M1} one can choose a generating hypersurface $Y_i \subset X_i$ that intersects the embedded torus $\mathcal{T}_i$ transversely into a knot $K_i$. Then a new generating hypersurface $Y$ in the fiber sum $X$ is obtained by splicing $Y_1$ and $Y_2$ along $K_1$ and $K_2$ with framing induced from that of $\mathcal{T}_1$ and $\mathcal{T}_2$ respectively. 

\subsubsection{Fiber Sum along Curves}

Let $\gamma \subset X$ be a simple closed curve in an integral homology $S^1 \times S^3$ so that $1_X \cdot [\gamma]=1$. Let $\nu(\gamma)$ be a tubular neighborhood, for which we fix a framing $\nu(\gamma) \cong S^1 \times D^3$. Note that $[S^1, SO(3)] \cong \Z/2$, there are two choices of framings. We denote by $M:=\cl (X \backslash \nu(\gamma))$ the closure of the complement of $\nu(\gamma)$ whose boundary is $\partial M =S^1 \times S^2$. 

\begin{dfn}
Let $(X_1, \gamma_1)$ and $(X_2, \gamma_2)$ be two sets of data as above. The fiber sum of $(X, \gamma_1)$ and $(X, \gamma_2)$ is the manifold given by 
\[
X_1 \#_{\gamma} X_2=M_1 \cup_{\id} M_2,
\]
where we have oriented $\partial M_1$ and $\partial M_2$ in a reversed manner. 
\end{dfn}

For $i=1, 2$ one can choose a generating hypersurface $Y_i \subset X_i$ so that $\gamma_i$ intersects $Y_i$ transversely and positively into a single point due to the fact that $1_{X_i} \cdot [\gamma_i]=1$. Then a new generating hypersurface $Y$ in the fiber sum $X$ can be chosen as the connected sum of $Y_1$ and $Y_2$ at the intersecting points with the embedded curves. 

\subsection{Seiberg-Witten Theory on $3$-Manifolds}\label{SWT3}
In this section we gather necessary information we will use of the Seiberg-Witten theory on $3$-manifolds from \cite{KM1}. 

Let $(Y, h)$ be a Riemannian closed smooth oriented $3$-manifold equipped with a spin$^c$ structure $\mathfrak{t}=(S, \rho)$, where $S$ is a $U(2)$-bundle, and $\rho: TY \to \End(S)$ is the Clifford multiplication. As before we have the space of connections $\mathcal{A}_k(Y, \mathfrak{t})$, the blown-up configuration space $\mathcal{C}^{\sigma}_{k}(Y, \mathfrak{t})$, the space of gauge transformations $\G_{k+1}(Y, \mathfrak{t})$, and the quotient blown-up configuration space $\mathcal{B}^{\sigma}_k(Y, \mathfrak{t})$. On $\mathcal{C}^{\sigma}_k(Y, \mathfrak{t})$ there is a vector field given by blowing-up the gradient vector field of the Chern-Simons-Dirac functional: 
\begin{equation}
\begin{split}
(\grad \mathcal{L})^{\sigma}(B, r, \psi):=(& -{1 \over 2} *F_{\B}-r^2 \rho^{-1}(\psi \psi^*)_0, \\
&-\Lambda(B, r, \psi)r, \\
& -D_B \psi +\Lambda(B, r, \psi) \psi), 
\end{split}
\end{equation}
where $D_B$ is the Dirac operator, and $\Lambda(B, r, \psi)=\langle \psi, D_B \psi \rangle_{L^2(Y)}$. Then one sees that the critical points of $(\grad \mathcal{L})^{\sigma}$ are one the following two types: 
\begin{enumerate}
\item[\upshape (i)] When $r \neq 0$, a critical point $(B, r, \psi)$ solves the equations
\begin{equation}
\begin{split}
-{1 \over 2} *F_{\B}-r^2 \rho^{-1}(\psi \psi^*)_0 & =0 \\
D_B \psi & =0.
\end{split}
\end{equation}
\item[\upshape(ii)] When $r=0$, a critical point $(B, 0, \psi)$ is characterized by the fact that  $B$ is a flat connection, i.e. $F_B=0$, and $\psi$ is an eigenvector of $D_B$. 
\end{enumerate}
A critical point is said to be irreducible if it's of the first type, and reducible if it's of the second type.

There is a Banach space of perturbations $\mathcal{P}(Y, \mathfrak{t})$ so that for each $\mathfrak{q} \in \mathcal{P}(Y, \mathfrak{t})$  we can perturb the vector field $(\grad \mathcal{L})^{\sigma}$ to $(\grad \mathcal{L}_{\mathfrak{q}})^{\sigma}$. Then we have the same description of the critical points for the perturbed vector field except all the equations in the description above are perturbed in a suitable way, and the Dirac operator is replaced by the perturbed one $D_{B, \mathfrak{q}}$ (see \cite[Section 10.3]{KM1}). A generic perturbation ensures that $\ker D_{B, \mathfrak{q}} =0$. Thus for a reducible critical point $\mathfrak{a}=(B, 0,\psi)$, the eigenvalue of $\psi$ is either positive or negative. We say $\mathfrak{a}$ is boundary-stable if it's positive, and boundary-unstable if it's negative. In this way we have decomposed the critical points of $(\grad \mathcal{L}_{\mathfrak{q}})^{\sigma}$ into three parts: 
\[ 
\mathfrak{C}(Y, \mathfrak{t})=
\mathfrak{C}^o(Y, \mathfrak{t}) \cup \mathfrak{C}^s(Y, \mathfrak{t}) \cup \mathfrak{C}^u(Y, \mathfrak{t})
\]
of irreducible, boundary-stable, and boundary-unstable critical points respectively. 

Let $[\mathfrak{a}], [\mathfrak{b}] \in \mathcal{B}^{\sigma}_k(Y, \mathfrak{t})$ be two critical points. Given a relative homotopy class $z \in \pi_1(\mathcal{B}^{\sigma}_k(Y, \mathfrak{t}); [\mathfrak{a}], [\mathfrak{b}])$, one can define a relative grading $\gr_z([\mathfrak{a}], [\mathfrak{b}])$ (c.f. \cite[Section 14.4]{KM1}) satisfying 
\begin{equation}
\gr_{z_1 \cdot z_2}([\mathfrak{a}], [\mathfrak{c}])=\gr_{z_1}([\mathfrak{a}], [\mathfrak{b}]) + \gr_{z_2}([\mathfrak{b}], [\mathfrak{c}]).
\end{equation}
Let's write $\M_z([\mathfrak{a}], [\mathfrak{b}])$ for the blown-up moduli space over the infinite cylinder $\R \times Y$ consisting of monopoles asymptotic to $[\mathfrak{a}]$ and $[\mathfrak{b}]$ in the negative and positive directions respectively. Then $\M_z([\mathfrak{a}], [\mathfrak{b}])$ is a smooth manifold of dimension either $\gr_z([\mathfrak{a}], [\mathfrak{b}])$, or $\gr_z([\mathfrak{a}], [\mathfrak{b}]) +1$, where the latter happens only when $[\mathfrak{a}]$ is boundary-stable, and $[\mathfrak{b}]$ is boundary-unstable (c.f. \cite[Proposition 14.5.7]{KM1}).

\begin{exm}\label{S3}
As an example, let's consider $(S^3, \mathfrak{t}_0)$ the $3$-sphere with its unique spin$^c$ structure. Since $S^3$ admits metrics with positive scalar curvature, there are no irreducible critical points. Up to gauge transformation there is a unique flat spin$^c$-connection $B_0$ on $S^3$. Applying generic perturbations one can assure that all eigenspaces of the Dirac operator $D_{B_0, \mathfrak{q}}$ have dimension $1$, which means up to gauge there is a unique eigenvector of unit length corresponding to each eigenvalue. Thus the critical-point set consists of a doubly infinite sequence corresponding to the spectral decomposition of $D_{B_0, \mathfrak{q}}$.We index the critical points so that $[\mathfrak{a}_{i}]$ corresponds to the $i+1$-th positive eigenvalue of $D_{B_0, \mathfrak{q}}$ when $i \geq 0$, and the $-i$-th negative eigenvalue when $i <0$. Moreover 
\begin{equation}
\gr_z([\mathfrak{a}_i], [\mathfrak{a}_{i-1}])=\left
\{ \begin{array}{ll}
 2 & \mbox{if $i \neq 1 $} \\
 1 & \mbox{if $i =1$.} 
  	\end{array}
	\right.
\end{equation}
\end{exm}

Let $W: Y_1 \to Y_2$ be a cobordism between two connected $3$-manifolds equipped with a spin$^c$ structure $\s$. We write $\mathfrak{t}_1=\s|_{Y_1}$, $\mathfrak{t}_2=\s|_{Y_2}$. Let $W_o$ be the manifold obtained from $W$ by attaching two cylindrical ends to its boundary. Given $[\mathfrak{a}] \in \mathfrak{C}(Y_1, \mathfrak{t}_1)$, $[\mathfrak{b}] \in \mathfrak{C}(Y_2, \mathfrak{t}_2)$, and $z \in \pi_0(\mathcal{B}^{\sigma}_k([\mathfrak{a}], W_o, [\mathfrak{b}]))$, one can also define a relative grading $\gr_z([\mathfrak{a}], W, [\mathfrak{b}])$ so that the moduli space $\M_z([\mathfrak{a}], W_o, [\mathfrak{b}])$ consisting of monopoles on $W_o$ asymptotic to $[\mathfrak{a}]$ and $[\mathfrak{b}]$ in the two directions respectively has dimension $\gr_z([\mathfrak{a}], W, [\mathfrak{b}])$ or $\gr_z([\mathfrak{a}], W, [\mathfrak{b}])+1$, where the latter happens only when $[\mathfrak{a}]$ is boundary-stable, and $[\mathfrak{b}]$ is boundary-unstable.  

\section{Counting Irreducible Monopoles}\label{CIM}

In section we present the first half of the proof of the main results by a neck-stretching argument. Let $X=M \cup N$ be an integral homology $S^1 \times S^3$ decomposed as in the process of forming either of the fiber sums, i.e. $N$ is a tubular neighborhood of either an embedded torus, or a simple closed curve as in Section \ref{set-up}, $M$ is the closure of its complement. We write $V= \partial M$, which is either $T^3$ or $S^1 \times S^2$. A neighborhood of $V$ in $X$ is identified with $(-1, 1) \times V$ so that $(-1, 0] \times V \subset M, [0, 1) \times V \subset N$. Let $h$ be a metric on $V$ which is either flat or has positive scalar curvature depending on $V=T^3$, or $S^1 \times S^3$. We consider metrics $g$ on $X$ satisfying 
\begin{enumerate}
\item[\upshape (i)] The restriction of $g$ to the neighborhood of $V$ is the  product metric:
\[
g|_{(-1, 1) \times V} = dt^2 +h.
\]
\item[\upshape (ii)] The restriction of $g$ on $N$ has nonnegative scalar curvature, and positive at some point. 
\end{enumerate}
We denote by $\Met(X, h)$ the set of such metrics. Given $T >0$, we stretch the neck $(-1, 1) \times V$ of $X$ to obtain $(X_T, g_T)$, where
\begin{equation}
X_T=M \cup [-T, T] \times V \cup N, \;  g_T|_{(-T, T) \times V} = dt^2 +h. 
\end{equation}
We also get $M_T = M \cup [0, T] \times V$, $N_T=[-T, 0] \times V \cup N$. Identifying $[-T, 0] \times V$ and $[0, T] \times V$, we see that $M_T \subset X_T$. Similarly $N_T \subset X_T$. We write the cylindrical-ended version as 
\[
M_o:=M \cup [0, \infty) \times V, \; N_o:= (-\infty, 0] \times V \cup N,
\]
and $X_o = M_o \cup N_o$ with metric $g_o$ naturally extended. As for perturbations, we would like to consider $\beta \in \mathcal{P}(X)$ of the form 
\begin{equation}
\beta = \beta_M + \beta_N, 
\end{equation}
where $\supp \beta_M \subset \Int M, \supp \beta_N \subset \Int N$. Over $X_T$, the perturbations are chosen to have the form $\beta_T  = \beta_{T, M} + \beta_{T, N}$ so that $\supp \beta_{T, M} \subset M_{T \over 2}$, and $\supp \beta_{T, N} \subset N_{T \over 2}$. 

\begin{dfn}\label{dadm}
We say a metric $g$ on $X$ is admissible with respect to the decomposition $X=M \cup N$ if the spin Dirac operator 
\begin{equation}\label{adm}
D^+(X_o, g_o): L^2_1(X_o, W^+) \to L^2(X_o, W^-)
\end{equation}
is an isomorphism. 
\end{dfn}
The main result we will prove in this section is 
\begin{prop}\label{cirr}
Given two sets of data $X_1=M_1 \cup N_1$ and $X_2=N_2 \cup M_2$ as above with $\partial M_1 =- \partial M_2=V$, we form $X=M_1 \cup M_2$. Suppose the metrics $g_1, g_2$ are admissible with respect to the decompositions. Moreover we choose the spin structures $\s_i$ on $X_i$ so that $\s_1|_{\partial M_1} = \s_2|_{\partial M_2}$. Then there exists $T_0 >0$ such that for all $T >T_0$ and regular pairs $(g_{i, T}, \beta_{i, T})$ of $X_{i, T}$ of the form discussed above, we have 
\begin{equation}
\# \M_{g_T, \beta_T}(X_T, \s)=\# \M_{g_{1, T}, \beta_{1, T}}(X_{1, T}, \s_1) + \# \M_{g_{2, T}, \beta_{2, T}}(X_{2, T}, \s_2),
\end{equation}
provided the induced pair $(g_T, \beta_T)$ is regular as well. 
\end{prop}

The strategy of the proof is to analyze the moduli space over the manifolds with cylindrical end and then apply the gluing theorem to count. To simplify notations we will drop the decorations of perturbations, metrics, and spin$^c$ structures when writing the moduli spaces unless they are relevant to the argument. 

Before proceeding to the proof, we explain here how the admissibility condition of the metrics can be achieved in our case. According to Theorem 10.3 in \cite{LRS} one can find a metric $g$ on $X$ in $\Met(X, h)$ such that $D^+(X_o, g_o)$ is an isomorphism if the spin Dirac operator on $Y$ has trivial kernel, i.e. $\ker D(V, h)=0$. When $V=S^1 \times S^2$, the metric $h$ has positive scalar curvature. Thus $\ker D(V, h)=0$ holds automatically. When $Y=T^3$ and $h$ is flat, the only spin structure whose Dirac operator has nontrivial kernel is given by the product of the spin structure on $S^1$ which does not extend over a disk. But none of the spin structures on $X$ restricts to this one. So $\ker D(V, h)=0$ in the case as well. 

Another thing we would like to mention here is the perturbations we are using in the neck-stretching process. In the statement of Proposition \ref{cirr} we only need the perturbations $\beta_T=\beta_{T, M}+\beta_{T, N}$. To make use of the critical points in Subsection \ref{SWT3} and arguments in \cite{KM1}, we add perturbations supported on $(-{T \over 2}, {T \over 2}) \times V \subset X_T$ of the form in \cite[Proposition 24.4.10]{KM1}, as well as the correponding ones on the cylindrical-ended manifold $X_o$. Over those closed manifolds we are counting irreducible monopoles with respect to regular pairs $(g, \beta)$, thus adding these extra perturbations with $L^2_k$ small size does not affect the counting. 

\subsection{Critical Points on $T^3$ and $S^1 \times S^2$}

Let $X=M \cup N$ be a decomposition of an integral homology $S^1 \times S^3$ as above. In order to apply the gluing argument and count monopoles we want to find the critical points $[\mathfrak{a}]$ on $V$ such that the moduli space $\M([\mathfrak{a}], N_o)$ consisting of monopoles on $(N_o, \s_N)$ asymptotic to $[\mathfrak{a}]$ has dimension $0$. Let $W: V \to S^3$ be the cobordism obtained by removing a $4$-ball in $N$. A standard gluing argument identifies the moduli spaces (see for example \cite[Lemma 27.4.2]{KM1})
\begin{equation}
\M([\mathfrak{a}], W_o, [\mathfrak{a}_0]) \cong \M([\mathfrak{a}], N_o),
\end{equation}
where $[\mathfrak{a}_0]$ is the critical point on $(S^3, \mathfrak{t}_0)$ correponding to the first positive eigenvalue of the Dirac operator as in Example \ref{S3}. Recall that $\dim \M([\mathfrak{a}], W_o, [\mathfrak{a}_0])= \gr_z ([\mathfrak{a}], W, [\mathfrak{a}_0])$. Thus we need to find those critical points $[\mathfrak{a}]$ of $(V, \mathfrak{t})$ such that 
$ \gr_z ([\mathfrak{a}], W, [\mathfrak{a}_0])=0.$

Now we reverse the role of $M$ and $N$, i.e. $X=N \cup M$  with $\partial N=V, \partial M=-V$. We would like to find the critical points $[\mathfrak{a}]$ on $V$ such that the moduli space $\M(N_o, [\mathfrak{a}])$ has dimension $0$. We denote by $W': S^3 \to V$ the cobordism obtained by removing a $4$-ball in $N$. Then the gluing theorem identifies the moduli spaces: 
\begin{equation}
\M([\mathfrak{a}_{-1}], W'_o, [\mathfrak{a}]) \cong \M(N_o, [\mathfrak{a}]),
\end{equation}
where $[\mathfrak{a}_{-1}]$ is the critical point on $(S^3, \mathfrak{t}_0)$ given by the first negative eigenvalue of the corresponding Dirac operator. Now we need to find critical points $[\mathfrak{a}]$ such that $\gr_z([ \mathfrak{a}_{-1}], W', [\mathfrak{a}])=0. $

\subsubsection{Critical Points on $T^3$}
 
Let $\mathfrak{t}$ be the spin$^c$ structure on $T^3$ with $c_1(\mathfrak{t})=0$. Thus the spinor bundle $S \to T^3$ is trivial. Let $[A_0]$ be the equivalence class of the trivial connection on $S$. The the space of equivalence classes of flat $\spin^c$-connections on $S$ is parametrized by the Picard torus
\[
\mathbb{T}:= H^1(T^3; i \R) / 2\pi i H^1(T^3; \Z).
\]
In \cite[Section 37]{KM1} Kronheimer-Mrowka showed that one can choose perturbations $\mathfrak{q}$ on $(T^3, \mathfrak{t})$ so that 
\begin{enumerate}
\item[\upshape (i)] There are no irreducible critical points. 
\item[\upshape (ii)] The reducible critical points $[A, 0, \psi]$ are specified as follows. $[A]-[A_0] \in \mathbb{T}$ is a critical point of a Morse function on $\mathbb{T}$ given by the perturbation $\mathfrak{q}$, $\psi$ is an eigenvector of the perturbed Dirac operator $D_{A, \mathfrak{q}}$.
\end{enumerate}
Following Kronheimer-Mrowka we label the critical points on $[A_0] +\mathbb{T}$ by 
\begin{equation}
\begin{array}{ccc}
 & w & \\
z^1 & z^2 & z^3 \\
y^1 & y^2 & y^3 \\
& x &
\end{array}
\end{equation}
so that $x=[A_0]$ is the minimal, $y^i$ has index $1$, $z^i$ has index $2$, and $w$ is the maximal. For each of these critical points on $[A_0]+\mathbb{T}$, we get a doubly infinite sequence of critical points for $(\grad \mathcal{L}_{\mathfrak{q}})^{\sigma}$ in $\mathcal{B}^{\sigma}_k(T^3, \mathfrak{t})$: 
\begin{equation}
\begin{array}{ccc}
 & w_i & \\
z^1_i & z^2_i & z^3_i \\
y^1_i & y^2_i & y^3_i \\
& x_i &
\end{array}
\end{equation}
indexed as in the case of $S^3$.

\begin{lem}\label{rg}
Let $W: T^3 \to S^3$ be the cobordism obtained by removing a $4$-ball in $N$. The relative gradings are computed as follows
\begin{enumerate}
\item[\upshape (i)] When $i \geq 0$, 
\begin{equation}
\gr_z([\mathfrak{a}], W, [\mathfrak{a}_0] )=\left
\{ \begin{array}{llll}
2i+2 & \mbox{if $[\mathfrak{a}]=w_i$} \\
2i+1 & \mbox{if $[\mathfrak{a}]=z^j_i$} \\
2i  & \mbox{if $[\mathfrak{a}]=y^j_i$}\\
2i+1 & \mbox{if $[\mathfrak{a}]=x_i$}.
  	\end{array}
\right.
\end{equation}
\item[\upshape (ii)] When $i <0$,  
\begin{equation}
\gr_z([\mathfrak{a}], W, [\mathfrak{a}_0] )=\left
\{ \begin{array}{llll}
2i+3 & \mbox{if $[\mathfrak{a}]=w_i$} \\
2i+2 & \mbox{if $[\mathfrak{a}]=z^j_i$} \\
2i+1  & \mbox{if $[\mathfrak{a}]=y^j_i$}\\
2i+2 & \mbox{if $[\mathfrak{a}]=x_i$}.
  	\end{array}
\right.
\end{equation}
\end{enumerate}
\end{lem}
\begin{proof}
Let's consider the moduli space $\M(W_o)$ consisting of monopoles $[A, r, \phi]$ of finite energy, i.e. 
\[
{1 \over 4}\int_{W_o} |F_{\A}|^2 +\int_{W_o} |\nabla_A r\phi|^2 + {1 \over 4} \int_{W_o} (|r\phi|^4 + s^2 |r\phi|^2) < \infty, 
\]
where $s$ is the scalar curvature on $W_o$. Then there are asymptotic maps \[
\partial_+: \M(W_o) \to \mathfrak{C}(S^3) \text{ and } \partial_-: \M(W_o) \to \mathfrak{C}(T^3)
\]
so that $\M([\mathfrak{a}], W_o, [\mathfrak{b}])=\{ [\Gamma] \in \M(W_o) : \partial_+ [\Gamma] = [\mathfrak{a}], \partial_- [\Gamma]=[\mathfrak{b}] \}$. Since $[\Gamma] \in \M(W_o)$ has finite energy, and the perturbations have been chosen so that the critical points are nondegenerate, $[\Gamma]$ has exponential decay on both ends. Thus $\partial_+[\Gamma]$ is boundary-stable, and $\partial_-[\Gamma]$ is boundary-unstable. Thus the top stratum of $\M(W_o)$ is given by $\M_z(w_{-1}, W_o, [\mathfrak{a}_0])$ due to the correspondence between the relative grading and the dimension of moduli spaces. 

On the other hand the formal dimension of the top stratum of $\M(W_o)$ is given by the index of the deformation complex of the Seiberg-Witten equation using weighted Sobolev space, which is given by 
\begin{equation}
d=b_1(W)-b_0(W)-b^+(W)+\ind D^+_A,
\end{equation}
where $[A, r, \phi]$ is a monopole in the top stratum. Note that the metrics that we allow on $N= D^2 \times T^2$ have positive scalar curvature in the interior, thus $\ind D^+_A=0$. We then conclude that 
\begin{equation}
\gr_z(w_{-1}, W, [\mathfrak{a}_0]) = \dim \M_z(w_{-1}, W_o, [\mathfrak{a}_0])= 1.
\end{equation}
The conclusion now follows from the additivity of the relative grading and the computation of the rational grading in \cite[Section 37.2]{KM1}.
\end{proof}

When $[\mathfrak{a}] \in \mathfrak{C}(T^3, \mathfrak{t})$ is boundary-unstable, each monopole in $\M([\mathfrak{a}], W_o, [\mathfrak{a}_0])$ has exponential decay on both ends. Thus the positivity of scalar curvature on $W_o$ implies that $\M([\mathfrak{a}], W_o, [\mathfrak{a}_0]) = \emptyset$. Combining Lemma \ref{rg} we conclude that the critical points $[\mathfrak{a}] \in \mathfrak{C}(T^3, \mathfrak{t})$ satisfying $\dim \M([\mathfrak{a}], W_o, [\mathfrak{a}_0])=0$ are 
\[
y^1_0, \; y^2_0, \; y^3_0,
\]
for which the counting of the moduli spaces of $\M([\mathfrak{a}], W_o, [\mathfrak{a}_0])$ can be computed as follows.

\begin{prop}\label{dim0}
Let $W: T^3 \to S^3$ be the cobordism as above. Then after a reordering of $z_0^j$, we have 
\[
\# \M([\mathfrak{a}], W_o, [\mathfrak{a}_0]) =\left
\{ \begin{array}{ll}
 1 & \mbox{if $[\mathfrak{a}]=y_0^1$} \\
 0 & \mbox{if $[\mathfrak{a}]=y_0^2, y_0^3$}.
  	\end{array}
\right.
\]
\end{prop}

\begin{proof}
We write $y_0^j=[A_j, 0, \phi_j], j=1, 2, 3$. Fixing a parametrization of $T^3 \cong S^1 \times S^1 \times S^1$, we let $\mathfrak{t}_j$ be the spin structure on $T^3$ given by the product of spin structures on $S^1$ such that the spin structure on the $j$-th $S^1$ is the one that extends over $D^2$, and on the other two $S^1$'s does not extend. Then one can choose perturbations so that $A_j$'s are given by the restriction of spin connections corresponding to $\mathfrak{t}_j$ respectively. We choose $N=D^2 \times T^2$ in the beginning. Thus 
\[
\# \M(y_0^j, W_o, [\mathfrak{a}_0])=\# \M(x_0, N_o)   =0, j=2, 3.
\]
Up to gauge transformation $A_1$ is the unique flat spin$^c$-connection on $N_o$ asymptotic to $y_0^1$. Note that $\M(y_0^1, W_o, [\mathfrak{a}_0])$ consists of reducible monopoles, and has dimension $0$, thus is identified with $\C P^0$ given by the kernel of the perturbed Dirac operator $D_{A_1, \mathfrak{q}}$. Since $\C P^0$ is complex, and the orientation on $\mathbb{T}$ is induced from the orientation of $N$, thus we conclude that 
$\# \M(y_0^1, W_o, [\mathfrak{a}_0])=1$. 
\end{proof}

\begin{rem} \label{dim0c}
If we consider the cobordism $W$ obtained by removing a $4$-ball from an manifold $N$ satisfying $H_*(N; \Z) \cong H_*(D^2 \times T^2; \Z)$, $\partial N=-T^3$, and $H_1(N; \Z) \to H_1(T^3; \Z)$ is injective, then the conclusion of Proposition \ref{dim0} holds for counting the reducibles $\# \M^{\Red}([\mathfrak{a}], W_o, [\mathfrak{a}_0])$ following the same argument. The same remark is applied to Proposition \ref{dim01}.
\end{rem}

Now identify $\partial N=T^3$ instead of $-T^3$, we let $W': S^3 \to T^3$ be the cobordism obtained by removing a $4$-ball in $N$. The analogue of Proposition \ref{dim0} is the following one. 

\begin{prop}\label{dim01}
Let $W': S^3 \to T^3$ be the cobordism as above. Then 
\begin{enumerate}
\item[\upshape (i)] The critical points on $(T^3, \mathfrak{t})$ satisfying $\dim \M ([\mathfrak{a}_{-1}], W'_o, [\mathfrak{a}])=0$ are 
\[
z^1_{-1}, \; z^2_{-1}, \; z^3_{-1}, x_{-1}.
\]
\item[\upshape (ii)] After a possible reordering the counting of the $0$-dimensional moduli spaces is given by 
\[
\# \M([\mathfrak{a}_{-1}], W'_o, [\mathfrak{a}]) =\left
\{ \begin{array}{ll}
 1 & \mbox{if $[\mathfrak{a}]=z^2_{-1}, z^3_{-1}$} \\
 0 & \mbox{if $[\mathfrak{a}]=z^1_{-1}, x_{-1}$}.
  	\end{array}
\right.
\]
\end{enumerate}
\end{prop}

\begin{proof}
The relative gradings $\gr_z([\mathfrak{a}_{-1}], W', [\mathfrak{a}])$ can be read from the rational gradings computed in \cite[Section 37]{KM1} and the definition of the rational grading \cite[Definition 28.3.1]{KM1}. When $[\mathfrak{a}]$ is boundary-stable, monopoles in $\M([\mathfrak{a}_{-1}], W_o', [\mathfrak{a}])$ have exponential decay on both ends. Thus it's empty due to the positivity of scalar curvature on $W'$. When $[\mathfrak{a}]$ is boundary-unstable, the rational gradings in \cite[(37.6b)]{KM1} give us that $z_{-1}^j, x_{-1}$ are the only $4$ critical points satisfying the dimension condition. 

We write $z^j_0=[A_j, 0, \phi_j]$, $j=1, 2, 3$. After fixing a parametrization of $T^3 \cong S^1 \times S^1 \times S^1$, one can choose perturbations so that $A_j$ is the spin connection of the spin structure on $T^3$ given by the product of spin structures on $S^1$ where on the $j$-factor it's the spin structure that does not extend over the disk, and on the other two $S^1$-factors extends over the disk. Then we see that  $\M([\mathfrak{a}_{-1}], W'_o, z^1_{-1})=\emptyset.$ Since the restriction map $H^1(N; \Z) \to H^1(T^3; \Z)$ is injective, once the flat connection $[A_j]$, $j=2, 3$, on the boundary is prescribed there is a unique flat connection up to gauge on $N$ restricting to $A_j$. Thus $\M([\mathfrak{a}_{-1}], W'_o, z^j_0)$, $j=2, 3$, is identified with $\C P^0$ as before, which gives us the counting. 

When $[\mathfrak{a}]=x_{-1}$, the connection corresponding to $x_{-1}$ is $A_0$, which do not extend to $N$. Thus $\M([\mathfrak{a}_{-1}], W'_o, x_{-1}) = \emptyset$. 
\end{proof}

\subsubsection{Critical Points on $S^1 \times S^2$} Now let $\mathfrak{t}$ be the unique torsion spin$^c$ structure on $S^1 \times S^2$. Here we deduce similar results for $(S^1 \times S^2, \mathfrak{t})$ as in the case of $T^3$. 

Let $[A_0]$ be the equivalence class of the trivial connection on the spinor bundle $S$. The Picard torus $\mathbb{T}$ is now a circle. As in \cite[Section 36]{KM1} one can choose perturbations $\mathfrak{q}$ on $(S^1 \times S^2, \mathfrak{t})$ so that 
\begin{enumerate}
\item[\upshape (i)] There are no irreducible critical points. 
\item[\upshape (ii)] Each reducible critical points $[A, 0, \psi]$ are given as follows. $[A]-[A_0] \in \mathbb{T}$ is a critical point of a Morse function on $\mathbb{T}$ given by the perturbation $\mathfrak{q}$, $\psi$ is an eigenvector of the perturbed Dirac operator $D_{A, \mathfrak{q}}$.
\end{enumerate}

Since $\mathbb{T}$ is a circle, one choose perturbations having two critical points, which we denote by $u, v$ such that $v=[A_0]$ is the minimal point, $u$ is the maximal point. Thus we get two doubly infinite sequences of critical points for $(\grad \mathcal{L}_{\mathfrak{q}})^{\sigma}$: $u_i, v_i$, which are indexed as before. 

\begin{lem}\label{rg2}
Let $W: S^1 \times S^2 \to S^3$ be the cobordism obtained by removing a $4$-ball in $N$. The relative gradings are computed as follows
\begin{enumerate}
\item[\upshape (i)] When $i \geq 0$, 
\begin{equation}
\gr_z([\mathfrak{a}], W, [\mathfrak{a}_0] )=\left
\{ \begin{array}{ll}
2i+1 & \mbox{if $[\mathfrak{a}]=u_i$} \\
2i & \mbox{if $[\mathfrak{a}]=v_i$}.
  	\end{array}
\right.
\end{equation}
\item[\upshape (ii)] When $i <0$,  
\begin{equation}
\gr_z([\mathfrak{a}], W, [\mathfrak{a}_0] )=\left
\{ \begin{array}{ll}
2i+2 & \mbox{if $[\mathfrak{a}]=u_i$} \\
2i+1 & \mbox{if $[\mathfrak{a}]=v_i$}.
  	\end{array}
\right.
\end{equation}
\end{enumerate}
\end{lem}

\begin{proof}
We apply the same argument as in the proof of Lemma \ref{rg}. Consider the moduli space $\M(W_o)$ consisting of monopoles of finite energy. Then the top stratum of $\M(W_o)$ is $\M_z(u_{-1}, W_o, [\mathfrak{a}_0])$. On the other hand the formal dimension of the top stratum is given by
\[
d=b_1(W)-b_0(W)-b^+(W)+\ind D_A^+=0,
\]
where the vanishing of $\ind D_A^+$ follows from the fact that the metrics on $W$ have positive scalar curvature. Thus we get 
\begin{equation}\label{2.14}
\gr_z(u_{-1}, W, [\mathfrak{a}_0]) = \dim \M_z(u_{-1}, W_o, [\mathfrak{a}_0])=0.
\end{equation}
Then the results for other critical points follow from the additivity and the computation in \cite[(36.1)]{KM1}. 
\end{proof}

From Lemma \ref{rg2} we see that all critical points $[\mathfrak{a}] \in \mathfrak{C}(S^1 \times S^2, \mathfrak{t})$ satisfying $\dim \M([\mathfrak{a}], W_o, [\mathfrak{a}_0])=0$ are 
\[
v_0 \text{ and } u_{-1}.
\]

\begin{prop}
Let $W: S^1 \times S^2 \to S^3$ be the cobordism as above. Then we have
\begin{equation}
\# \M([\mathfrak{a}], W_o, [\mathfrak{a}_0]) =\left
\{ \begin{array}{ll}
 1 & \mbox{if $[\mathfrak{a}]=v_0$} \\
 0 & \mbox{if $[\mathfrak{a}]=u_{-1}$}.
  	\end{array}
\right.
\end{equation}
\end{prop}

\begin{proof}
When $[\mathfrak{a}]=u_{-1}$, $u_{-1}$ being boundary-unstable implies that monopoles in $\M_z(u_{-1}, W_o, [\mathfrak{a}_0])$ have exponential decay. Since the metrics on $W$ have positive scalar curvature, the top stratum consisting of irreducibles is empty. Thus 
\[
\# \M(u_{-1}, W_o, [\mathfrak{a}_0]) =0.
\]

When $[\mathfrak{a}]=v_0$, the moduli space $\M_z(v_0, W_o, [\mathfrak{a}_0])$ consists entirely of reducibles over a single flat connection, which is identified as $\C P^0$. Thus 
\[
\# \M(v_0, W_o, [\mathfrak{a}_0]) =1.
\]
\end{proof}

As in the end of last subsection, we identify $\partial N=S^1 \times S^2$ instead of $-S^1 \times S^2$. Removing a $4$-ball in $N$ gives us a cobordism $W': S^3 \to S^1 \times S^2$. We can also prove the analogue of Proposition \ref{dim01} in the case of $S^1 \times S^2$.

\begin{prop}
Let $W': S^3 \to S^1 \times S^2$ be the cobordism as above. Then 
\begin{enumerate}
\item[\upshape (i)] The critical points on $(S^1 \times S^2, \mathfrak{t})$ satisfying $\dim \M ([\mathfrak{a}_{-1}], W'_o, [\mathfrak{a}])=0$ are 
\[
v_0 \text{ and } u_{-1}
\]
\item[\upshape (ii)] The counting of the $0$-dimensional moduli spaces is given by 
\[
\# \M([\mathfrak{a}_{-1}], W'_o, [\mathfrak{a}]) =\left
\{ \begin{array}{ll}
 1 & \mbox{if $[\mathfrak{a}]=u_{-1}$} \\
 0 & \mbox{if $[\mathfrak{a}]=v_0$}.
  	\end{array}
\right.
\]
\end{enumerate}
\end{prop}

\begin{proof}
Similar as in (\ref{2.14}) we get 
\begin{equation}
\gr_z([\mathfrak{a}_{-1}], W', v_0) =0.
\end{equation}
Then there are only two critical points $v_0$ and $u_{-1}$ satisfying the dimension condition.

When $[\mathfrak{a}]$ is boundary-stable, we know that $\M([\mathfrak{a}_{-1}], W'_o, [\mathfrak{a}]) = \emptyset$. Thus 
\[
\# \M([\mathfrak{a}_{-1}], W'_o, v_0)=0.
\]
We argue as before to identify $\M([\mathfrak{a}_{-1}], W'_o, u_{-1})$ with $\C P^0$. Thus 
\[
\# \M([\mathfrak{a}_{-1}], W'_o, u_{-1})=1.
\]
\end{proof}

\subsection{Counting Irreducibles}
This section is devoted to the proof of Proposition \ref{cirr}. 

\begin{prop}\label{count1}
Let $X=M \cup N$ be a decomposition of an integral homology $S^1 \times S^3$ as before, where $\partial M=-\partial N=V$. Suppose the spin Dirac operator 
\begin{equation}\label{iso}
D^+(X_o, g_o): L^2_1(X_o, W^+) \to L^2(X_o, W^-)
\end{equation}
is an isomorphism. 
Then there exists $T_1 >0$ such that for all $T > T_1$, and regular pairs $(g_T, \beta_T)$ of $X_T$ with $\beta_T \to \beta_o$, we have 
\begin{equation}
\# \M_{g_T, \beta_T}(X_T, \s) =
\left
\{ \begin{array}{ll}
 \# \M_{g_o, \beta_o}(M_o, y_0^1) & \mbox{if $V=T^3$} \\
 \# \M_{g_o, \beta_o}(M_o, v_0) & \mbox{if $V=S^1 \times S^2$}. 
  	\end{array}
	\right.
\end{equation}
\end{prop}

\begin{proof}
We only prove the case when $V=T^3$. The other case is the same. Following Proposition 26.1.4 in \cite{KM1} one can consider the compactified moduli space over $(X_o, \s)$: 
\begin{equation}\label{cption}
\M^+(X_o, \s):= \bigcup_{[\mathfrak{a}], [\mathfrak{b}]} \M(M_o, [\mathfrak{a}]) \times \breve{\M}^+([\mathfrak{a}], [\mathfrak{b}]) \times \M([\mathfrak{b}], N_o),
\end{equation}
where $\breve{\M}^+([\mathfrak{a}], [\mathfrak{b}])$ is the compactification of the unparametrized moduli space over $\R \times V$ as in \cite[Theorem 16.1.3]{KM1} consisting of unparametrized broken trajectories. Theorem 9.1 in \cite{LRS} provides us with a homeomorphisms 
\begin{equation}
\rho_{T}: \M(X_T, \s) \to \M^+(X_o, \s)
\end{equation}
for all $T$ greater than a sufficiently large $T_1>0$. Thus $ \# \M(X_T, \s) = \# \M^+(X_o, \s)$. 

Since $\dim \M^+(X_o, \s) =0$, only the top stratum is nonempty. Moreover each factor in the product of (\ref{cption}) has dimension equal to $0$. If $[\mathfrak{a}] \neq [\mathfrak{b}]$, \cite[Proposition 26.1.6]{KM1} tells us that $[\mathfrak{a}]$ and $[\mathfrak{b}]$ have to be boundary-stable and boundary-unstable respectively. From Proposition \ref{dim0} we conclude that $\M([\mathfrak{b}], N_o) = \emptyset$ when $\dim \M([\mathfrak{b}], N_o) =0$ and $[\mathfrak{b}]$ is boundary-unstable. Thus $[\mathfrak{a}] = [\mathfrak{b}]$, which has to be boundary-stable. Recall from Proposition \ref{dim0} when $V=T^3$ the only critical point making $\M([\mathfrak{a}], N_o) \neq \emptyset$ is $y_0^1$. Thus we conclude 
\[
\#\M^+(X_o, \s) = \# (\M(M_o, y_0^1) \times \M(y_0^1, N_o)) =  \# \M(M_o, y_0^1).
\]
\end{proof}

If we reverse the role of $M$ and $N$ we get an analogous result as follows. 

\begin{prop}\label{count2}
Let $X=N \cup M$ be a decomposition of an integral homology $S^1 \times S^3$ as before, where $\partial N=-\partial M=V$. Suppose the spin Dirac operator 
\begin{equation}\label{iso2}
D^+(X_o, g_o): L^2_1(X_o, W^+) \to L^2(X_o, W^-)
\end{equation}
is an isomorphism. 
Then there exists $T_2 >0$ such that for all $T > T_2$, and regular pairs $(g_T, \beta_T)$ of $X_T$ with $\beta_T \to \beta_o$, we have 
\[
\# \M_{g_T, \beta_T}(X_T, \s) =
\left
\{ \begin{array}{ll}
 \# \M(z^2_{-1}, M_o) +  \# \M(z^3_{-1}, M_o) & \mbox{if $V=T^3$} \\
 \# \M(u_{-1}, M_o) & \mbox{if $V=S^1 \times S^2$}. 
  	\end{array}
	\right.
\]
\end{prop}

\begin{proof}
We only prove the case when $V= T^3$. The other case is proved in the same way. As in the proof of Proposition \ref{count1} we consider the compactified moduli space 
\begin{equation}\label{cption}
\M^+(X_o, \s):= \bigcup_{[\mathfrak{a}], [\mathfrak{b}]} \M(N_o, [\mathfrak{a}]) \times \breve{\M}^+([\mathfrak{a}], [\mathfrak{b}]) \times \M([\mathfrak{b}], M_o).
\end{equation}
Moreover for all $T$ greater than a sufficiently large $T_2 >0$, due to (\ref{iso2}) we identify 
\[
\# \M(X_T, \s) = \# \M^+(X_o, \s).
\]
Analyzing as before we see that in the union $[\mathfrak{a}]=[\mathfrak{b}]$. By Proposition \ref{dim01} to get nonzero counting of $\M(N_o, [\mathfrak{a}]) \cup \M([\mathfrak{a}], M_o)$, it's necessary that $[\mathfrak{a}]=z^2_{-1}$, or $z^3_{-1}$. In both cases $\# \M(N_o, [\mathfrak{a}])=1$. Thus 
\[
\# \M^+(X_o, \s)= \M(z^2_{-1}, M_o) +  \# \M(z^3_{-1}, M_o).
\]
\end{proof}

Now we are ready to give a proof for Proposition \ref{cirr}. 

\begin{proof}[Proof of Proposition \ref{cirr}]
We only prove the case when $V=T^3$. The argument for the case of $S^1 \times S^2$ is the same, and simpler in details. We consider the compactified moduli space 
\begin{equation}\label{cptionm}
\M^+(X_o, \s):= \bigcup_{[\mathfrak{a}], [\mathfrak{b}]} \M(M_{1, o}, [\mathfrak{a}]) \times \breve{\M}^+([\mathfrak{a}], [\mathfrak{b}]) \times \M([\mathfrak{b}], M_{2, o}).
\end{equation}
The invertibility of $D^+(X_o, g_o)$ enables us to apply the gluing theorem \cite[Theorem 9.1]{LRS} to identify $\# \M(X_T, \s)= \# \M^+(X_o, \s)$ for $T$ is greater than some $T_0 >0$ large enough. Now we need to study the compactified space $\M^+(X_o, \s)$ in more details.

When $[\mathfrak{a}] \neq [\mathfrak{b}]$, due to the fact that $\dim \M^+(X_o, \s)=0$ $[\mathfrak{a}]$ and $[\mathfrak{b}]$ have to be boundary-stable and boundary-unstable respectively. In this case $ \breve{\M}^+([\mathfrak{a}], [\mathfrak{b}]) = \emptyset $ following from Lemma 36.1.1 in \cite{KM1}.

When $[\mathfrak{a}]=[\mathfrak{b}]$, each component in $\M^+(X_o, \s)$ has the form $\M(M_{1, o}, [\mathfrak{a}]) \times \M([\mathfrak{a}], M_{2, o}).$ Thus 
\[
\# \M^+(X_o, \s)= \sum_{[\mathfrak{a}]} \# \M(M_{1, o}, [\mathfrak{a}]) \cdot \# \M([\mathfrak{a}], M_{2, o}).
\]

If $[\mathfrak{a}]$ is boundary-stable, to get nonzero counting $[\mathfrak{a}]=y^1_0$ from Remark \ref{dim0c}. In this case $\M(y^1_0, M_{2, o})$ has dimension $0$, and consists of reducible monopoles. Note that the restriction map $H^1(M_2; \R) \to H^1(\partial M_2; \R)$ is injective. Thus up to gauge there is a unique flat connection on $M_2$ restricting to the connection given by $y^1_0$. Due to dimension reason, we can identify $\M(y^1_0, M_{2, o})$ with $\C P^0$ so that $\# \M(y^1_0, M_{2, o})=1$. 

If $[\mathfrak{a}]$ is boundary-unstable, Remark \ref{dim0c} implies that only when $[\mathfrak{a}]=z^2_{-1}$, or $z^3_{-1}$ can we get nonzero counting. The same argument as above gives us that 
\[
\# \M(M_{1, o}, z^j_{-1})=1, \; j=2, 3.
\]
Thus we conclude that 
\begin{equation}
\begin{split}
\# \M^+(X_o, \s) & = \# \M(M_{1, o}, y^1_0) + \# \M(z^2_{-1}, M_{2, o}) + \# \M(z^3_{-1}, M_{2, o}) \\
                            & = \# \M(X_{1, T}, \s_1) + \# \M(X_{2, T}, \s_2),
\end{split}
\end{equation}
where the second line of the equality follows from Proposition \ref{count1} and Proposition \ref{count2}.
\end{proof}

\section{Comparing Index Correction Terms}\label{CICT}

The second half of the proof of the main results is to compare the index correction terms of the manifolds before and after applying fiber sums. The main ingredient of the proof is the excision principle over end-periodic manifolds proved in \cite[Section 5]{M1}, which we briefly recall here. 

Let $(Z_+, \s)$ be a spin$^c$ $4$-manifold obtained from an integral homology $S^1 \times S^3$ which we denote by $X$ as before. Suppose we can decompose $Z_+$ as a union of two end-periodic open sets, i.e. $Z_+=P \cup Q$, where $P$ and $Q$ restricted to the end $W_+$ are invariant under the $\Z$-translation. Moreover we want $P \cap Q = (-1, 1) \times \tilde{V} \subset Z_+$ is a tubular neighborhood of an embedded end-periodic $3$-manifold $\tilde{V}$ where the restriction of the metric $g$ on $Z_+$ is of product form. We write $V \subset X$ for the projection of the end of $\tilde{V}$ to the $4$-manifold $X$. Since the metric is a product, we can insert cylinders $[-T, T] \times \tilde{V}$ and $[-T, T] \times V$ to $Z_+$ and $X$ respectively. We denote the results by $Z_{+, T}$ and $X_T$. 

Suppose now we are given two sets of such data $Z_{1, +}=P_1 \cup Q_1$, $Z_{2, +}=P_2 \cup Q_2$ as above together with perturbed Dirac operators (as in (\ref{pdo}))
\begin{align*}
D^+_{\beta_1}(Z_{1, +}) : & L^2_1(Z_{1, +}, W^+_1) \to L^2(Z_{1, +}, W^-_1) \\
D^+_{\beta_2}(Z_{2, +}) : & L^2_1(Z_{2, +}, W^+_2) \to L^2(Z_{2, +}, W^-_2).
\end{align*}
We say those two sets of data are excisable if the spin$^c$ structures on the overlaps $(-1, 1) \times \tilde{V}_1$ and $(-1, 1) \times \tilde{V}_2$ are identified as well as the twisted Dirac operators:
\[
D^+_{\beta_1}((-1, 1) \times \tilde{V}_1) \cong D^+_{\beta_2}((-1, 1) \times \tilde{V}_2). 
\]
Then we can form two new end-periodic manifolds 
\[
\tilde{Z}_{1, +}=P_1 \cup Q_2, \; \tilde{Z}_{2, +} = P_2 \cup Q_1,
\]
and perturbed Dirac operators $D^+_{\tilde{\beta}_1}(\tilde{Z}_{1, +})$, $D^+_{\tilde{\beta}_2}(\tilde{Z}_{1, +})$. At last we assume that all the Dirac operators on the end-periodic manifolds are Fredholm. A criterion for elliptic operators on an end-periodic manifold extending as a Fredholm operator from $L^2_1$ to $L^2$ is derived in \cite[Lemma 4.3]{T1}. In our case it's equivalent to the regularity of the pairs $(g_i, \beta_i)$ (see also \cite{MRS}). 

\begin{thm}(\cite[Theorem 5.4]{M1})
Given two sets of excisable data as above, we assume that the metrics $g_1$ and $g_2$ are both admissible in the sense of Definition \ref{dadm}. Then there exists $T_3 >0$, and perturbations $\beta_{i, T}$ converging to $\beta_{i, o}$ such that for all $T > T_3$ one has 
\[
\ind D^+_{\beta_{1, T}}(Z_{1, +, T}) + \ind D^+_{\beta_{2, T}}(Z_{2, +, T}) = \ind D^+_{\tilde{\beta}_{1, T}}(\tilde{Z}_{1, +, T}) +  \ind D^+_{\tilde{\beta}_{2, T}}(\tilde{Z}_{2, +, T}) 
\]
\end{thm}

Now we can use the excision principle to compare the index correction terms. We follow the set-up in Section \ref{CIM}. 

\begin{prop}\label{aic1}
Given two sets of data $X_1=M_1 \cup N_1$ and $X_2=N_2 \cup M_2$ as in Proposition \ref{cirr} with $\partial M_1= -\partial M_2=T^3$, we form the fiber sum $X=M_1 \cup M_2$. Suppose the metrics $g_1 \in \Met(X_1, h)$, $g_2 \in \Met(X_2, h)$ are both admissible. Moreover we choose the spin structures $\s_i$ on $X_i$ so that $\s_1|_{\partial M_1} = \s_2|_{\partial M_2}$. Then there exists $T_4 >0$ such that for all $T > T_4$ and regular pairs $(g_{i, T}, \beta_{i, T})$ of $X_{i, T}$ we have 
\begin{equation}
\omega(X_T, g_T, \beta_T)=\omega(X_{1, T}, g_{1, T}, \beta_{1, T}),
\end{equation}
provided the induced pair $(g_T, \beta_T)$ is regular as well. 
\end{prop}

\begin{proof}
We will drop the stretching parameter $T$, perturbations, and metrics in the notation. Whenever we apply the excision principle we need a larger parameter. There are only finitely many times of excision, which gives us a lower bound $T_4$. Once metrics are chosen for $X_1$ and $X_2$, all other metrics are induced then. Whenever we need regularity, we choose generic perturbations $\beta_i$ to achieve this. For each fixed $T$, only finitely many steps are needed to achieve regularity. Being regular is an open condition, thus we can choose $\beta_{i, T}$ to vary continuously with respect to $T$. 

Recall $N_1$ is a tubular neighborhood of an embedded torus $\mathcal{T}_1 \subset X_1$. From the text below Definition \ref{dfn1}, we choose the generating hypersurface $Y_1 \subset X_1$ intersecting the torus $\mathcal{T}_1$ transversely into a knot $K_1$. Now we attach a $2$-handle to $K_1$ with framing given by the identification $N_1 \cong D^2 \times T^2$, and then take a spin $4$-manifold to cap off the boundary. In this way we have found a spin $4$-manifold $(Z_1, \s_1)$ with spin boundary $(Y_1, \mathfrak{t}_1)$ so that $K_1$ bounds a disk $D_1$, which is the core of the attaching $2$-handle, inside $Z_1$. Now we remove a small $4$-ball centered at the center of $D_1$. Applying the excision principle to the excisable data 
\[
(Z_{1, +}=D^4 \cup Z_+ \backslash D^4, D^+(Z_{1, +})), \; (\bar{D}^4_-=(-\infty, 0] \times S^3 \cup \bar{D}^4, D^+(\bar{D}^4_-)),
\]
where $\bar{D}_-^4$ is the orientation-reversed $4$-ball attached by a negative cylindrical end $(-\infty, 0] \times S^3$, which is equipped with the unique spin structure, we get 
\begin{equation}
\ind D^+(Z_{1, +}) +\ind D^+(\bar{D}^4_-) = \ind D^+ (Z^c_{1, +}) + \ind D^+(S^4),
\end{equation}
where $Z^c_{1, +}=(- \infty, 0] \times S^3 \cup Z_{1, +} \backslash D^4$. Due to the positivity of scalar curvature on $D^4$ and $S^4$, we get 
\begin{equation}
\ind D^+(Z_{1, +})= \ind D^+(Z^c_{1, +}). 
\end{equation}
Repeat the whole process to $(X_2, \mathcal{T}_2, \s_2)$ we get 
\begin{equation}
\ind D^+(Z_{2, +}) = \ind D^+(Z^c_{2, +}). 
\end{equation}

Now the punctured disk $D^c_1: \mathcal{U} \to K_1$ is a concordance from the unknot in $S^3$ to $K_1$ in $Y_1$. Let's denote by $\pi_1: W_{1, +} \to X_1$ be the projection of the periodic end on $Z_{1, +}$ and by $N_{K_1}$ a tubular neighborhood of $(-\infty, 0] \times U_1 \cup D^c_1$. Let $Q_1=N_{K_1} \cup \pi_1^{-1}(N_1) \subset Z^c_{1, +}$. We then get a decomposition $Z^c_{1, +}=P_1 \cup Q_1$. Repeat the process to $(Z^c_{2, +}, K_2)$ we get a decomposition $Z^c_{2, +}=Q_2 \cup P_2$. 

Applying the excision principle to the excisable data 
\[
(Z^c_{1, +}=P_1 \cup Q_1, D^+(Z^c_{1, +})), \; (Z^c_{2, +}=Q_2 \cup P_2, D^+(Z^c_{2, +})),
\]
we get 
\begin{equation}
\ind D^+(Z^c_{1, +}) + \ind D^+(Z^c_{2, +})= \ind D^+(Z^c_+) + \ind D^+(\tilde{Q}), 
\end{equation}
where $Z^c_+=P_1 \cup P_2$, $\tilde{Q}=Q_1 \cup Q_2$. From the construction $\tilde{Q}$ is diffeomorphic to $\R \times S^1 \times S^2$ equipped with metrics of positive scalar curvature. Thus 
\begin{equation}
\ind D^+(Z^c_{1, +}) + \ind D^+(Z^c_{2, +})= \ind D^+(Z^c_+).
\end{equation}
Note that $Z^c_+$ has a cylindrical end of the form $(-\infty, 0] \times S^1 \times S^2$. Apply one more excision to this end with respect to the pair $S^1 \times S^3 = \overline{S^1 \times D^3} \cup S^1 \times D^3$, we get $Z_+$ with a cylindrical end modeled on the fiber sum $X$ such that
\begin{equation}
\ind D^+(Z^c_+) = \ind D^+(Z_+),
\end{equation}
where the vanishing of the other terms are caused the fact that $S^1 \times S^3$ admits an orientation-reversing diffeomorphism. Thus we get 
\begin{equation}
\begin{split}
\ind D^+(Z_{1, +}) + \ind D^+(Z_{2, +}) & =\ind D^+(Z^c_{1, +}) + \ind D^+(Z^c_{2, +}) \\
 & = \ind D^+(Z^c_+)=\ind D^+(Z_+). 
\end{split}
\end{equation}

Let $Z$ be the compact part of $Z_+$. From the construction we know that $Z$ is obtained as follows. We first remove a $4$-ball $D_i$ in $Z_i$ to get a cobordism $W^c_i: S^3 \to Y_i$ with an embedded concordance $D^c_i: \mathcal{U} \to K_i$ for each $i=1, 2$, then remove neighborhoods of $D^c_i$ in $W^c_i$, and glue the resulting manifolds accordingly to get a cobordism $W^c: S^1 \times S^2 \to Y$, finally fill $W^c$ with $S^1 \times D^3$ to get $Z$. By the additivity of signature, the manifolds we removed or filled in the first, second, and fourth steps all have signature $0$. Thus the gluing in the third step gives the additivity of signature:
\begin{equation}
\sigma(Z_1) + \sigma(Z_2) = \sigma(Z).
\end{equation}
\end{proof}

We also prove the corresponding additivity of the correction term in the case of fiber summing along curves. 

\begin{prop}\label{aic2}
Given two sets of data $X_1=M_1 \cup N_1$ and $X_2=N_2 \cup M_2$ as in Proposition \ref{cirr} with $\partial M_1= -\partial M_2=S^1 \times S^2$, we form the fiber sum $X=M_1 \cup M_2$. Suppose the metrics $g_1 \in \Met(X_1, h)$, $g_2 \in \Met(X_2, h)$ are both admissible. Moreover we choose the spin structures $\s_i$ on $X_i$ so that $\s_1|_{\partial M_1} = \s_2|_{\partial M_2}$. Then there exists $T_4 >0$ such that for all $T > T_4$ and regular pairs $(g_{i, T}, \beta_{i, T})$ of $X_{i, T}$ we have 
\begin{equation}
\omega(X_T, g_T, \beta_T)=\omega(X_{1, T}, g_{1, T}, \beta_{1, T}),
\end{equation}
provided the induced pair $(g_T, \beta_T)$ is regular as well. 
\end{prop}

\begin{proof}
For the same reason as in the proof of Proposition \ref{aic1} we will drop the neck-stretching parameter $T$, metrics, and perturbations in the notation. 

Now $N_1$ is a tubular neighborhood of an embedded curve $\gamma_1$ which intersects the generating hypersurface $Y_1$ positively into a single point. Choose a spin $4$-manifold $(Z_1, \s_1)$ with spin boundary $(Y_1, \mathfrak{t}_1)$, we remove a neighborhood $D_1$ of the intersection point in $Z_1$. Let $P_1=D_1 \cup \pi_1^{-1}(N_1) \subset Z_{1, +}$, and write the closure of the complement of $P_1$ to be $Q_1$. We repeat the construction on $(X_2, \gamma_2)$ to get a decomposition $Z_{2, +}=Q_2 \cup P_2$. 

Applying the excision principle to the excisable data 
\[
(Z_{1, +}=P_1 \cup Q_1, D^+(Z_{1, +})) \text{ and } (Z_{2, +}=Q_2 \cup P_2, D^+(Z_{2, +}))
\]
gives us 
\[
\ind D^+(Z_{1, +}) + \ind D^+(Z_{2, +}) = \ind D^+(Z_+) + \ind D^+(D^4_+),
\]
where $D^4_+ \cong D^4 \cup [0, \infty) \times S^3$, $Z_+=Q_1 \cup Q_2$ has a periodic end modeled on the fiber sum $X$. Since $D^4_+$ admits metrics of positive scalar curvature, we conclude 
\begin{equation}
\ind D^+(Z_{1, +}) + \ind D^+(Z_{2, +}) = \ind D^+(Z_+).
\end{equation}
Note that $Z$ is the boundary sum of $Z_1$ and $Z_2$, thus we get the additivity of the signature: 
\begin{equation}
\sigma(Z_1) + \sigma(Z_2) = \sigma(Z).
\end{equation}
\end{proof}

\section{Examples}\label{app}

In this section we discuss some examples which illustrate how the fiber sums we have been considering naturally arise.

\begin{exm}
Let $S^3_1(K)$ be the $3$-manifold obtained by performing $1$-surgery along a knot $K$ in $S^3$. The meridian of $K$ in $S^3$ becomes a knot $K'$ in $S^3_1(K)$. Then we get an embedded torus $\mathcal{T}_K=S^1 \times K' \subset S^1 \times S^3_1(K)$. Let $(X, \mathcal{T})$ be another pair in the construction of fiber sums along tori. The fiber sum formula gives us that 
\begin{equation}
\lambda_{SW}(X \#_{\mathcal{T}} S^1 \times S^3_1(K)) = \lambda_{SW}(X) - {1 \over 2} \Delta''_K(1),
\end{equation}
where $\Delta_K(t)$ is the symmetric Alexander polynomial of $K$. Here we have used the fact that 
\[
\lambda_{SW}(S^1 \times S^3_1(K)) = -\lambda (S^3_1(K))=-{1 \over 2} \Delta''_K(1). 
\]
\end{exm}

\begin{exm}
Consider the $n$-fold branched cover $\Sigma_n(K)$ of a knot $K$ in $S^3$. Denote by $\tau: \Sigma_n(K) \to \Sigma_n(K)$ the generating covering transformation of order $n$. We denote by $X_n(K)$ the mapping torus of $\Sigma_n(K)$ under this map:
\[
X_n(K)=[0, 1] \times \Sigma_n(K) / (0, y) \sim (1, \tau(y)).
\]
It's not hard to show that $X_n(K)$ has the same integral homology as $S^1 \times S^3$. Note that in \cite{LRS1} the Casson-Seiberg-Witten invariant of $X_n(K)$ is computed as 
\[
\lambda_{SW}(X_n(K)) = -{1 \over 8} \sum_{m=0}^{n-1} \sign^{m/n}(K),
\]
where $\sign^{m/n}(K)$ is the Tristram-Levine signature (see \cite[Section 6]{RS1}). In particular when $n=2$, we get 
\[
\lambda_{SW}(X_2(K))=-{1 \over 8} \sigma(K), 
\]
where $\sigma(K)$ is the knot signature. 

Let $\tilde{K} \subset \Sigma_n(K)$ be the branching set. Pick up a point $y_0 \in \tilde{K}$, there is a natural embedded curve $\gamma_K \subset X_n(K)$ given by closing the first factor $[0, 1]$ in the mapping torus at $y_0$. There is a preferred framing of $\gamma_K$ given as follows. Let's choose a $\tau$-invariant neighborhood $U_{y_0}$ of $y_0$ in $\Sigma_n(K)$, which we identify as $[0, 1] \times D^2$ where $\tau$ acts as
\[
\tau(s, re^{i \theta}) \mapsto (s, re^{i (\theta+ {2 \pi \over n})}).
\]
Then a tubular neighborhood $\nu(\gamma_K) \subset [0, 1] \times \Sigma_n(K) / \sim$ of $\gamma_K$ is identified as 
\begin{equation}
\begin{split}
\nu(\gamma_K) & \longrightarrow S^1 \times I \times D^2 \\
[t, s, re^{i\theta}] & \longmapsto ([t], s, re^{i( \theta+ t {2 \pi \over n})}).
\end{split}
\end{equation}

On the other hand, $\mathcal{T}_K=S^1 \times \tilde{K}$ gives us an embedded torus in $X_n(K)$, which is endowed with a preferred framing as follows. We identify a $\tau$-invariant neighborhood of $\tilde{K}$ with $S^1 \times D^2$ from that of $K$ in $S^3$, where $\tau$ acts as 
\[
\tau(e^{is}, re^{i\theta}) = (e^{is}, re^{i(\theta+ {2\pi \over n})}). 
\]
Then a tubular neighborhood $\nu(\mathcal{T}_K)$ is identified as 
\begin{equation}
\begin{split}
\nu(\mathcal{T}_K) & \longmapsto S^1 \times S^1 \times D^2 \\
[t, e^{is}, re^{i\theta}] & \longmapsto ([t], e^{is}, re^{i(\theta+t{2\pi \over m})}). 
\end{split}
\end{equation}

\begin{cor}Let $(X_n(K), \gamma_K)$ and $(X_n(K), \mathcal{T}_K)$ be the pairs as above. 
\begin{enumerate}
\item[\upshape (i)] Let $(X, \gamma)$ be a pair as in forming the fiber sum along curves. Then 
\[
\lambda_{SW}(X \#_{\gamma} X_n(K)) = \lambda_{SW}(X)-{1 \over 8} \sum_{m=0}^{n-1} \sign^{m/n}(K).
\]
\item[\upshape (ii)] Let $(X, \mathcal{T})$ be a pair as in forming the fiber sum along tori. Then 
\[
\lambda_{SW}(X \#_{\mathcal{T}} X_n(K)) = \lambda_{SW}(X)-{1 \over 8} \sum_{m=0}^{n-1} \sign^{m/n}(K).
\]
\end{enumerate}
\end{cor}
\end{exm}

\begin{exm}
Let $Y_1$ and $Y_2$ be two integral homology spheres. We fix basepoints $y_1 \in Y_1$ and $y_2 \in Y_2$. Let $c \hookrightarrow Y_1$ be a simple closed curve based at $y_1$ representing a non-central element in $\pi_1(Y_1, y_1)$. The preferred framing on $c$ identifies a tubular neighborhood of $c$ as $\nu(c) \cong S^1 \times D^2$. Let $A \subset D^2$ be the annulus with outer radius $1$, and inner radius ${1 \over 2}$. Consider a map on $S^1 \times A \subset \nu(c)$ given by 
\begin{align*}
f'_c: S^1 \times A & \longrightarrow S^1 \times A \\
(s, re^{i\theta}) & \longmapsto (s+ \rho(r), re^{i\theta}),
\end{align*}
where $\rho: [0, 1] \to [0, 1]$ is a smooth decreasing function satisfying $\rho(1)=0, \rho(r)=1$ for $r \leq {1 \over 2}$ with vanishing derivatives on both ends. Then we extend this map to $Y_1$ as identify on the rest, which we still denote by $f'_c$. Connect-summing $Y_1$ and $Y_2$ at $y_1$ and $y_2$ respectively on balls of radius $\epsilon$ with $0 < \epsilon < {1 \over 2}$, we can then extend the map $f'_c$ to 
\[
f_c: Y_1 \# Y_2 \longrightarrow Y_1 \# Y_2,
\]
which can be visualized as dragging $Y_2$ along the curve $c$. We denote the mapping torus $Y_1 \# Y_2$ under the map $f_c$ by $X_c$. Note that the map $f_c'$ is homotopic to the identity map via 
\begin{align*}
f'_{c, t}: S^1 \times D^2 & \longrightarrow S^1 \times D^2 \\
(s, re^{i\theta}) & \longmapsto (s+ 2t\rho(r), re^{i\theta}),
\end{align*}
where $t \in [0, {1 \over 2}]$. 

Let $\alpha \subset Y_1$ be the curve $S^1 \times \{y_1\}$. Then a representative $\gamma_1$ of the class $[\alpha] * [c] \in \pi_1(S^1 \times Y_1)$ can be choosen as follows. Write $t \in [0, 1]$ for the coordinate on the $S^1$-factor of $S^1 \times Y_1$. The basepoint of $S^1 \times Y_1$ is choosen as $(0, y_1)$. Under the identification $\nu(c) \cong S^1 \times D^2$, $y_1$ is gven by $(0, 0)$. We pick up $y'_1 \in \nu(c)$ so that $y'_1$ is identified with $(0, 1)$. Take an arc connecting $y_1$ to $y'_1$ given by $a(t)=(0, 2t)$, $t \in [0, {1 \over 2}]$, in $\nu(c)$. Then $\gamma_1$ lying in $S^1 \times \nu(c)$ is given by 
\[
\gamma_1(t) =\left
\{ \begin{array}{ll}
 f'_{c, t}(a(t)) & \mbox{if $t \in [0, {1 \over 2}]$ } \\
 a(1-t) & \mbox{if $ t \in [{1 \over 2}, 1]$.} 
  	\end{array}
	\right.
\]
Let $\gamma_2=S^1 \times \{y_2\}$ be the curve in $Y_2$. Then one sees that the mapping torus $X_c$ is the fiber sum of $(S^1 \times Y_1, \gamma_1)$ and $(S^1 \times Y_2, \gamma_2)$. Thus our result implies that 
\begin{equation}\label{isw}
\lambda_{SW}(X_c)=\lambda(Y_1) + \lambda (Y_2).
\end{equation}
In particular when $[c] \in \pi_1(Y_1, y_1)$ is of inifinite order, the map $f_c$ has infinite order in the mapping class group of $Y_1 \# Y_2$ (c.f. \cite{HL}). Thus (\ref{isw}) gives an example of computing the Casson-Seiberg-Witten invariant in the case of a mapping torus formed by an infinite order diffeomorphism. Note that the computation for a large class of mapping torus given by finite order diffeomorphisms was carried out in \cite{LRS1}. 

We note here that in principle the mapping torus $X_c$ is different from the product $X=S^1 \times Y_1 \# Y_2$. One can see this from their fundamental groups. From the construction we know that 
\[
\pi_1(X_c)= \langle \mathfrak{h}, \mathfrak{a} \in \pi_1(Y_1), \mathfrak{b} \in \pi_1(Y_2) | \mathfrak{h} \mathfrak{a} \mathfrak{h}^{-1} = \mathfrak{c} \mathfrak{a} \mathfrak{c}^{-1}, \mathfrak{h} \mathfrak{b} \mathfrak{h}^{-1}= \mathfrak{b} \rangle,
\]
where $\mathfrak{c}=[c] \in \pi_1(Y_1)$. On the other hand 
\[
\pi_1(X)= \langle \mathfrak{h} \rangle \oplus (\pi_1(Y_1) * \pi_1(Y_2)). 
\]
If both $\pi_1(Y_1)$ and $\pi_1(Y_2)$ have trivial center, then $\pi(X_c)$ has trivial center, but $\pi_1(X)$ does not. 
\end{exm}


\bibliographystyle{plain}
\bibliography{References}

\end{document}